\documentclass[reqno]{amsart}

\usepackage{verbatim}
\usepackage[textsize=scriptsize]{todonotes}
\usepackage{tikz-cd}
\usepackage{etoolbox}
\usepackage{etex}
\usepackage[hyperfootnotes=false]{hyperref}
\usepackage[multiple]{footmisc}
\usepackage{cleveref}
\usepackage[T1]{fontenc}
\usepackage{enumitem}

\usepackage{chemarr}
\usepackage{amssymb}
\usepackage{amsmath}
\usepackage{comment}
\usepackage{cleveref}
\usepackage{mathtools}
\usepackage{rotating}
\usepackage{wrapfig}
\usepackage{outlines}
\usepackage{graphicx}

\theoremstyle{definition}
\newtheorem{nul}{}[section]

\newtheorem{rmk}[nul]{Remark}

\newtheorem{exm}[nul]{Example}

\newtheorem*{dfn*}{Definition}
\newtheorem*{axm*}{Axiom}
\newtheorem*{ntn*}{Notation}
\newtheorem*{exm*}{Example}
\newtheorem*{exr*}{Exercise}
\newtheorem*{int*}{Intuition}
\newtheorem*{qst*}{Question}
\newtheorem*{rmk*}{Remark}

\theoremstyle{plain}

\newtheorem{thm}[nul]{Theorem}
\newtheorem{prop}[nul]{Proposition}
\newtheorem{lem}[nul]{Lemma}

\newtheorem{cor}[nul]{Corollary}

\newtheorem*{thm*}{Theorem}
\newtheorem*{prop*}{Proposition}
\newtheorem*{cor*}{Corollary}
\newtheorem*{lem*}{Lemma}
\newtheorem*{cnj*}{Conjecture}

\let\oldwidetilde\widetilde
\protected\def\widetilde{\oldwidetilde}

\DeclareMathOperator{\gr}{\mathrm{gr}}

\DeclareMathOperator{\can}{\mathrm{can}}

\DeclareMathOperator{\F}{\mathbb{F}}
\DeclareMathOperator{\bF}{\mathbb{F}}
\DeclareMathOperator{\bZ}{\mathbb{Z}}

\DeclareMathOperator{\bQ}{\mathbb{Q}}

\DeclareMathOperator{\rN}{\mathrm{N}}
\DeclareMathOperator{\rF}{\mathrm{F}}

\DeclareMathOperator{\cO}{\mathcal{O}}

\newcommand{\TP}{\mathrm{TP}}

\newcommand{\Z}{\mathbb{Z}}

\def\Zp{\mathbb{Z}_p}

\def\O{\mathcal{O}}

\def\Q{\mathbb{Q}}

\usepackage{relsize}
\usepackage[bbgreekl]{mathbbol}
\usepackage{amsfonts}

\DeclareSymbolFontAlphabet{\mathbbl}{bbold}

\def\Prism{\mathbbl{\Delta}}
\def\Prismhat{\widehat{\Prism}}

\usepackage{stmaryrd}

\DeclarePairedDelimiter\abs{\lvert}{\rvert}%
\makeatletter
\let\oldabs\abs
\def\abs{\@ifstar{\oldabs}{\oldabs*}}

\usepackage{tikz}
\usetikzlibrary{matrix,arrows,decorations}
\usepackage{tikz-cd}

\usepackage{adjustbox}

\let\oldtocsection=\tocsection
 
\let\oldtocsubsection=\tocsubsection
 
\let\oldtocsubsubsection=\tocsubsubsection
 
\renewcommand{\tocsection}[2]{\hspace{0em}\oldtocsection{#1}{#2}}
\renewcommand{\tocsubsection}[2]{\hspace{1em}\oldtocsubsection{#1}{#2}}
\renewcommand{\tocsubsubsection}[2]{\hspace{2em}\oldtocsubsubsection{#1}{#2}}

\usepackage{etoolbox}
\newtoggle{draft}
\toggletrue{draft}

\iftoggle{draft} {
\usepackage[margin=1.5in]{geometry}
\newcommand{\NB}[1]{\todo[color=gray!40]{#1}}
\newcommand{\TODO}[1]{\todo[color=red]{#1}}
}{ 
\usepackage[margin=1.5in]{geometry}
\newcommand{\NB}[1]{}
\newcommand{\TODO}[1]{}
\renewcommand{\todo}[1]{}
\renewcommand{\todo}[1]{}
}
\title{Exact bounds for even vanishing of $K_*(\bZ/p^n)$}

\author{Achim Krause}
\address{Department of Mathematics, University of Oslo, Oslo, Norway}
\email{achimkr@uio.no}
\author{Andrew Senger}
\address{Department of Mathematics, Harvard University, Cambridge, MA, USA}
\email{senger@math.harvard.edu}

\begin{document}
\begin{abstract}
  In this note, we prove that $K_{2i} (\Z/p^n) \neq 0$ if and only if $p-1$ divides $i$ and $0 \leq i \leq (p-1) p^{n-2}$, refining the even vanishing theorem of Antieau, Nikolaus and the first author in this case.
  As a corollary of our proof, we determine that the nilpotence order of $v_1$ in $\pi_* K(\Z/p^n)/p$ is equal to $\frac{p^n-1}{p-1}$.

  Our proof combines the recent crystallinity result for reduced syntomic cohomology of Hahn, Levy and the second author with the explicit complex computing the syntomic cohomology of $\O_K /\varpi^n$ constructed by Antieau, Nikolaus and the first author.
\end{abstract}
\maketitle

\setcounter{tocdepth}{1}
\tableofcontents
\vbadness 5000
	

\section{Introduction} \label{sec:intro}
In \cite{kzpn}, the first author together with Antieau and Nikolaus proves the following result:

\begin{thm}[{\cite[Theorem 1.4]{kzpn}}, Even vanishing theorem]
Let $K$ be a finite extension of $\bQ_p$ with ramification index $e$, $\cO_K$ its ring of integers, $\varpi\in \cO_K$ a uniformizer, and $n>0$.
Then for 
  \[
    i-1\geq\frac{p}{p-1} \left(\frac{p}{p-1}(p^{\lceil \tfrac{n}{e}\rceil} -1) - p^{\lceil \tfrac{n}{e} \rceil} (\lceil \tfrac{n}{e}\rceil-\tfrac{n}{e})\right),
  \]
  we have $K_{2i-2}(\cO_K/\varpi^n)=0$.
\end{thm}

In the unramified case $e=1$ this bound simplifies to $i-1\geq \left(\frac{p}{p-1}\right)^2 (p^n-1)$. As indicated by the tables in \cite{kzpn}, this bound is not sharp. The purpose of this note is to determine sharp bounds in the case $K = \Q_p$. We prove:

\begin{thm} \label{thm:main}
  We have $K_{2i}(\Z/p^n)\neq 0$ if and only if $p-1\mid i$ and $i\leq (p-1)p^{n-2}$.
\end{thm}


To prove \Cref{thm:main}, we first note that since the groups $K_*(\Z/p^n)$ are finitely generated abelian groups, it follows by Quillen's computation of $K_* (\F_p)$ \cite{QuillenK} and Gabber rigidity \cite{Gabber} that it suffices to prove the theorem for the $p$-completed groups $\pi_* K (\Z/p^n)^{\wedge} _p$.
To access these, we make use of the equivalence $K (\Z / p^n)^{\wedge} _p \simeq \tau_{\geq 0} TC(\Z / p^n)$ coming from the Dundas--Goodwillie--McCarthy theorem \cite{DGM}
and the motivic spectral sequence of Bhatt--Morrow--Scholze \cite{BMS, BS, APC}:
\[H^j (\Z_p (i) (R)) \Rightarrow \pi_{2i-j} TC(R).\]

In the case of interest, we have $H^j (\Z_p (i) (\Z / p^n)) = 0$ when $j > 2$ and $H^0 (\Z_p (i) (\Z / p^n)) = 0$ when $i > 0$, so that for $i \geq 2$ we have $K_{2i-2} (\Z / p^n)^{\wedge} _p \cong H^2 (\Z_p (i) (\Z / p^n))$.
This reduces us to the study of the syntomic cohomology groups $H^2 (\Z_p (i) (\Z / p^n))$.
Since these are finitely generated $\Z_p$-modules, they vanish and only if their mod $p$ reduction
\[H^2 (\Z_p (i) (\Z / p^n))/ p \cong H^2 (\F_p (i) (\Z / p^n))\]
does.

Our first tool to study these groups is the following theorem of Hahn and Levy with the second author. The proof relies on a new crystallinity property for reduced syntomic cohomology.

\begin{thm}[{{\cite[Theorem 7.1]{HLS}}}] \label{thm:surjective}
  For $n \geq 2$, the map $H^2 (\F_p (i) (\Z_p)) \to H^2 (\F_p (i) (\Z / p^n))$ is surjective.
\end{thm}

By \cite[Theorem 1.5]{LiuWang}, there are classes $v_1 \in H^0 (\F_p (p-1) (\Zp))$, $\lambda_1 \in H^1 (\F_p (p) (\Zp))$ and $\partial \in H^1 (\F_p (0) (\Zp))$ so that
\[H^2 (\F_p (*) (\Zp)) \cong \F_p [v_1] \{\lambda_1 \partial\}.\]
We will also reprove this theorem using the techniques of \cite{kzpn} in \Cref{sec:Zp-comp}.

As a consequence, even vanishing reduces to the question of the $v_1$-torsion order of $\lambda_1 \partial$.
\Cref{thm:main} therefore reduces to the following result:

\begin{prop} \label{thm:v1-tors}
  The $v_1$-torsion order of $\lambda_1 \partial$ in $H^2 (\F_p (*) (\Z/p^n))$ is equal to $p^{n-2}$.
  In other words, $v_1 ^{p^{n-2}} \lambda_1 \partial = 0$ and $v_1 ^{p^{n-2}-1} \lambda_1 \partial \neq 0$.
\end{prop}

This result will be proven in \Cref{cor:vanishing} and \Cref{prop:nonvanishing}.
The nonvanishing $v_1 ^{p^{n-2}-1} \lambda_1 \partial \neq 0$ is an immediate consequence of \cite[Theorem 1.1]{HLS}, which implies that the map $\F_p (*) (\Z) /v_1 ^{p^{n-2}} \to \F_p (*) (\Z/p^n) / v_1 ^{p^{n-2}}$ admits a section.
We will prove the vanishing $v_1 ^{p^{n-2}} \lambda_1 \partial = 0$ in \Cref{sec:final} by an explicit analysis of $\bF_p(*)(\Z/p^n)$ using the description of syntomic cohomology of rings of this form deduced in \cite{kzpn}.


Let us conclude with some corollaries of \Cref{thm:v1-tors}.
First, we see that the bounds of \cite[Theorem 1.1(1)]{HLS} are optimal, at least in the unramified setting.

\begin{cor}
  The functor of animated $p$-complete rings
  \[R \mapsto \F_p (*) / v_1 ^{k}\]
  does not factor through the derived mod $p^n$ reduction
  \[R \mapsto R / p^n\]
  when $k > p^{n-2}$. 
\end{cor}

\begin{proof}
  If there were such a factorization, then the map
  $\F_p (*) (\Z_p) /v_1 ^{k} \to \F_p (*) (\Z/p^n) / v_1 ^{k}$
  would admit a section,
  which contradicts the relation $v_1 ^{p^{n-2}} \lambda_1 \partial = 0$ of \Cref{thm:v1-tors} when $k > p^{n-2}$.
\end{proof}

Moreover, when $p \geq 5$ we are able to determine the exact order of nilpotence of $v_1$ in $\pi_* K(\Z/p^n)/p$, since our determination of $H^2 (\F_p (*) (\Z/p^n))$ rules out the possibility of hidden extensions.

\begin{cor}
  Suppose that $p \geq 5$.
  Then the order of nilpotence of $v_1$ in 
  $\pi_* K (\Z / p^n)/p$ and $\pi_* TC (\Z / p^n)/p$
  is exactly equal to $[n]_p = \frac{p^n-1}{p-1}$.
  In other words, we have $v_1 ^{[n]_p} = 0$ and $v_1 ^{[n]_p-1} \neq 0$
  in these homotopy rings.
\end{cor}

\begin{proof}
  By \cite[Theorem 1.8]{kzpn}, we know that the order of nilpotence of $v_1 \in H^0 (\F_p (p-1) (\Z/p^n))$ is equal to $[n]_p$.
  It therefore follows that, if $v_1 ^{[n]_p} \neq 0 \in \pi_* TC (\Z / p^n)/p$,
  it must be detected in filtration $2$ in the motivic spectral sequence.
  But we have $\abs{v_1 ^{[n]_p}} = 2(p^n-1) = \abs{v_1 ^{[n]_p -1} \lambda_1 \partial}$.
  Since $[n]_p -1 \geq p^{n-2}$, it follows from \Cref{thm:surjective} and \Cref{thm:v1-tors} that the the $2$-line is $0$ in this degree, so we are done.

  Finally, the result for $K$-theory follows from the equivalence $K (\Z / p^n)/p \simeq \tau_{\geq 0} TC (\Z / p^n)/p$, which is a consequence of the Dundas--Goodwillie--McCarthy theorem.
\end{proof}

Finally, \Cref{thm:v1-tors} is the key ingredient in proving the following result of \cite{HLS}, which allows one to transport the results about syntomic cohomology in \cite{HLS} to results about $TC$ and $K$-theory.

\begin{cor}[{{\cite[Corollary 7.6]{HLS}}}]
  If $p = 2$, suppose that $k$ is divisible by $4$.
  Then the motivic spectral sequence
  $H^j (\F_p (i) (\Z/p^n)) / v_1 ^k \Rightarrow \pi_{2i-j} TC(\Z / p^n) / (p, v_1^k)$
  degenerates at the $E_2$-page.
\end{cor}

\begin{rmk}
  Our results can be extended from the case of $\Z/p^n$ to the unramified case $W(k) / p^n$, where $k$ is a finite field, using the same techniques. We have restricted ourselves to the case of $\Z/p^n$ here for the sake of simplicity of exposition, as well as because the computations of \cite{HLS} are only carried out in the case $\Z/p^n$.
\end{rmk}

\subsection{Acknowledgements}
During the course of this work, Andrew Senger was supported by NSF grant DMS-2103236. Achim Krause was supported during part of this work by the Deutsche Forschungsgemeinschaft (DFG, German Research Foundation) – project-id 427320536 – SFB 1442, as well as under Germany’s excellence strategy EXC 2044 390685587, Mathematics Münster: Dynamics–Geometry–Structure. Achim Krause also gratefully acknowledges support from the Institute for Advanced Study.

\section{Syntomic cohomology computations}
\subsection{Recollections from \cite{kzpn}}
In \cite{kzpn}, an explicit complex for the syntomic cohomology of rings of the form $\O_K$ or $\O_K / \varpi^n$ is given, where $\O_K$ is a characteristic $(0,p)$ discrete valuation ring with finite residue field $\F_q$, and $\varpi$ is a uniformizer for $\O_K$.
We begin by recalling the structure of this complex, then discuss how it interacts with the multiplicative structure on $\Z_p (*) (R)$.

\begin{thm}
  For $R = \O_K$ or $\cO_K/\varpi^n$, $\bZ_p(i)(R)$ is equivalent to the total complex of a square
\[
\begin{tikzcd}
  \rN^{\geq i} \Prismhat^{(1)}_{R/W(\bF_q)\llbracket z\rrbracket}\{i\}\rar{\nabla}\dar{\can-\varphi} & \rN^{\geq i} \Prismhat^{(1),\nabla}_{R/W(\bF_q)\llbracket z\rrbracket}\{i\}\dar{\can-\varphi^\nabla}\\
  \Prismhat^{(1)}_{R/W(\bF_q)\llbracket z\rrbracket}\{i\}\rar{\nabla} & \Prismhat^{(1),\nabla}_{R/W(\bF_q)\llbracket z\rrbracket}\{i\}
\end{tikzcd}
\]
\end{thm}

Here, we have for $R = \O_k /\varpi^n$:
\begin{enumerate}
  \item All terms carry an additional filtration (``$\rF$'') that they are complete with respect to.
  \item $\Prismhat^{(1)}_{R/W(\bF_q)\llbracket z\rrbracket}$ is $\rF$-completely free over $W(\bF_q)$ on monomials $z^k\prod_u f_u^{e_u}$ where $k<n$ and $e_u$ is a sequence with $e_u<p$ for all $u$ and $e_u=0$ for all but finitely many $u$. Here $z$ has $\rF$-filtration $1$ and $f_u$ has $\rF$-filtration $np^u$. $\Prismhat^{(1)}_{R/W(\bF_q)\llbracket z\rrbracket}\{i\}$ is a free module of rank $1$ over the non Breuil-Kisin-twisted version, on a generator we call $t^{-i}$. (This is called $w(s^i)$ in \cite{kzpn}, arising from a Breuil-Kisin orientation $s$ of non-Frobenius twisted prismatic cohomology, but the notation $t^{-1}$ chosen here highlights the relationship to the element $t^{-1}$ in the Tate spectral sequence computing topological periodic homology $\TP_{-2}$).
  \item $\rN^{\geq j}\Prismhat^{(1)}_{R/W(\bF_q)\llbracket z\rrbracket}\{i\}$ is $\rF$-completely free over $W(\bF_q)$ on the elements $E(z)^{\max(j-\sum p^ue_u,0)} z^k\prod_u f_u^{e_u}t^{-i}$ where $k<n$ and $e_u$ is a sequence with $e_u<p$ for all $u$ and $e_u=0$ for all but finitely many $u$. Here, the symbol $E(z)$ is related to a choice of an Eisenstein polynomial for $\O_K$.
  \item $\Prismhat^{(1),\nabla}_{R/W(\bF_q)\llbracket z\rrbracket}\{i\}$ is isomorphic to $\Prismhat^{(1)}_{R/W(\bF_q)\llbracket z\rrbracket}\{i\}$, with a shift in $\rF$-filtration ($\rF^{\geq k}\Prismhat^{(1),\nabla}\{i\}\cong \rF^{\geq k-1} \Prismhat^{(1)}\{i\}$.)
\item $\rN^{\geq j}\Prismhat^{(1),\nabla}_{R/W(\bF_q)\llbracket z\rrbracket}\{i\}$ is isomorphic to $\rN^{\geq j-1}\Prismhat^{(1)}_{R/W(\bF_q)\llbracket z\rrbracket}\{i\}$ (with the same shift in $\rF$-filtration). We think of $\Prismhat^{(1),\nabla}\{i\}$ as a free $\Prismhat^{(1)}_{R/W(\bF_q)\llbracket z\rrbracket}$-module on the symbol $\nabla z t^{-i}$ of Nygaard and $\rF$-filtration $1$, but warn that neither the Frobenius map $\varphi^\nabla$ nor the product structure on $\bZ_p(*)(R)$ behave intuitively with respect to this notation. 
\end{enumerate}

For $R=\cO_K$, the description is analogous, but $\Prismhat^{(1)}_{R/W(\bF_q)\llbracket z\rrbracket}$ is instead free on monomials $z^k$ with no condition on $k$ (and no $f_u$).

There is a multiplicative structure on $\bZ_p(*)(R)$, which we will need to understand in terms of the above complex. For this, let us write $N_s = \rN^{\geq *}\Prismhat^{(1)}_{R/W(\bF_q)\llbracket z_0,\ldots,z_s\rrbracket}\{*\}$ and $P_s = \Prismhat^{(1)}_{R/W(\bF_q)\llbracket z_0,\ldots,z_s\rrbracket}\{*\}$. In \cite{kzpn} the rows of the above square are identified up to quasi-isomorphism with the total complexes of the cosimplicial diagrams $N_\bullet$ and $P_\bullet$, using a map $\theta: P_1 \to \Prismhat^{(1),\nabla}_{R/W(\bF_q)\llbracket z\rrbracket}\{i\}$. The vertical maps correspond to the difference of the maps $\can,\varphi: N_s\to P_s$. Rewriting equalizers as totalisation of a cosimplicial diagram by right Kan extension, we obtain a bicosimplicial ring
\[
  \begin{tikzcd}
    N_0\dar[shift right,"{(\mathrm{id},\can)}",swap]\dar[shift left,"{(\mathrm{id},\varphi)}"] \rar[shift left] \rar[shift right] & N_1 \lar\dar[shift right,"{(\mathrm{id},\can)}",swap]\dar[shift left,"{(\mathrm{id},\varphi)}"] \rar[shift left=2]\rar[shift right=2]\rar & \ldots\lar[shift right=1]\lar[shift left=1]\\
    N_0\times P_0\dar[shift left=2]\dar\uar\dar[shift right=2] \rar[shift left]\rar[shift right] & N_1 \times P_1 \uar\lar\dar[shift left=2]\dar\dar[shift right=2] \rar[shift left=2]\rar[shift right=2]\rar&  \ldots\lar[shift right=1]\lar[shift left=1]\\
    \vdots\uar[shift left]\uar[shift right] & \vdots \uar[shift left]\uar[shift right]& \\
  \end{tikzcd}
\]
with quasi-isomorphism to the original square obtained by projection to the second factor in the second row, and using the map $\theta$ in the second column. For a bicosimplicial ring $R_{\bullet,\bullet}$, the product structure on its total complex is given by the (iterated) Alexander-Whitney diagonal, where for $\alpha\in R_{n,m}$, $\beta\in R_{n',m'}$, we have
\[
  \alpha\beta = (-1)^{mn'} \sigma_* \alpha \cdot \rho_* \beta,
\]
where $\sigma: [n]\times [m] \to [n+m]\times [n'+m']$ is the ``front face'' inclusion in each factor, and $\rho: [n']\times [m']\to [n+m]\times [n'+m']$ analogously the ``back face'' inclusion in each factor.

\begin{rmk}
  \label{rem:products}
We observe the following special cases:
\begin{enumerate}
  \item If $\alpha \in N_0$ is a cycle in the above bicosimplicial diagram, and $\beta$ is arbitrary in bidegree $(n,m)$, $\alpha\beta$ is represented by multiplying $\beta$ with the image of $\alpha$ in bidegree $(n,m)$ (which is unique by virtue of $\alpha$ being a cycle). Since the map $\theta$ used to identify the bicosimplicial diagram with the above square in \cite{kzpn} is linear over $d^0: P_0 \to P_1$, this means that in the square, multiplication by a cycle $\alpha$ in bidegree $(0,0)$ is really computed pointwise in the obvious way.
  \item Writing $\partial$ for the cycle in bidegree $(0,1)$ represented by $(0,1)\in N_0\times P_0$, and letting $\beta\in N_s$ be an arbitrary element of bidegree $(s,0)$, we have that $\partial \beta$ is represented by $(0,(-1)^s\varphi(\beta))$, with $\varphi$ arising from the back-face map. Since under $\theta$, $\varphi$ corresponds to $\varphi^\nabla$ in the second column, we get that in the original square, multiplication with $\partial$ corresponds to applying $\varphi$ and $-\varphi^\nabla$. (Note that if $\beta$ is a cycle, these also agree with $\pm\can(\beta)$).
  \item A $1$-cycle in the square which is purely supported in bidegree $(1,0)$ can be lifted to a $1$-cycle in the bicosimplicial diagram which is purely supported in this bidegree. Since bidegree $(2,0)$ maps to zero in the square, this shows that elements purely supported in bidegree $(1,0)$ multiply to $0$ with each other. Analogously, elements purely supported in bidegree $(0,1)$ multiply to zero.
\end{enumerate}
Note that due to the non-linearity of the map $\theta$ used to identify the cosimplicial resolutions $N_\bullet$ and $P_\bullet$ with the shorter complexes involving $\nabla$, general products of classes in bidegrees $(0,1)$ and $(1,0)$ with each other are a bit subtle. We will however only require the above observations.
\end{rmk}

\begin{rmk}
For $j>0$, the associated graded $\gr^j_{\rF}$ of each term in the above square is free of rank $1$ over $W(\bF_q)$. For $j=0$, the left hand terms have $\gr^0_{\rF}$ free of rank $1$ over $W(\bF_q)$, while the right hand terms have vanishing $\gr^0_{\rF}$. In particular, all terms are $p$-torsion free, and the mod $p$ reduction $\bF_p(i)(R)$ is equivalent to the total complex of the mod $p$ reduction of the above square.
\end{rmk}

\subsection{Computations in the case $R = \Z_p$}\label{sec:Zp-comp}
We now specialize to the case $\cO_K=\bZ_p$ and let the Eisenstein polynomial be given by $E(z)=z+p$.
We give a computation of $\F_p (*) (\Zp)$ in terms of the square of \cite{kzpn}, in particular identifying representatives for the classes $\partial$, $v_1$, and $\lambda_1$
that will be used in \Cref{sec:final} below to prove that $v_1^{p^{n-2}} \partial \lambda_1 = 0$ in $\F_p (*) (\Z/p^n)$.

To start, we collect some formulas:

\begin{lem}
  In the square for $\bF_p(i)(\Zp)$, we have:
  \begin{enumerate}
    \item $\can(z^kE(z)^it^{-i}) = z^{k+i}t^{-i}$
    \item $\varphi(z^kE(z)^it^{-i}) = z^{pk}t^{-i}$
    \item $\nabla(z^kt^{-i}) = kz^{k-1}\nabla z t^{-i}$ mod $\rF^{\geq k+1}$
    \item $\nabla(z^kE(z)^it^{-i}) = 0$ mod $\rF^{\geq k+1}$
    \item $\can(z^{k-1}E(z)^{i-1}\nabla z t^{-i}) = z^{k+i-2}\nabla z t^{-i}$
    \item $\varphi^\nabla(z^{k-1}E(z)^{i-1}\nabla z t^{-i}) = z^{pk-1}\nabla z t^{-i}$ mod $\rF^{\geq pk+1}$
  \end{enumerate}
\end{lem}
\begin{proof}
  The first statement follows from $E(z)=z+p$, and the second from $\varphi(t^{-1}) = \frac{t^{-1}}{\varphi(E(z))}$ (compare \cite[Remark 4.9]{kzpn}, noting the Frobenius twist here). For the fifth statement note the last sentence in the proof of \cite[Corollary 4.44]{kzpn}. The third and fourth statements are obtained from the proof of \cite[Lemma 4.61]{kzpn}, noting that the unit that appears there is $1$ mod $p$. Finally, for the last statement, consider the commutative diagram
  \[
    \begin{tikzcd}
      \gr^k_{\rF} \rN^{\geq i} \Prismhat^{(1)}_{\bZ_p/\bZ_p\llbracket z\rrbracket}\{i\} \rar{\nabla}\dar{\varphi} & \gr^k_{\rF} \rN^{\geq i} \Prismhat^{(1),\nabla}_{\bZ_p/\bZ_p\llbracket z\rrbracket}\{i\}\dar{\varphi^\nabla} \\ 
      \gr^{pk}_{\rF} \Prismhat^{(1)}_{\bZ_p/\bZ_p\llbracket z\rrbracket}\{i\} \rar{\nabla} & \gr^{pk}_{\rF} \Prismhat^{(1),\nabla}_{\bZ_p/\bZ_p\llbracket z\rrbracket}\{i\}.
    \end{tikzcd}
  \]
  By \cite[Lemma 4.61]{kzpn}, the bottom horizontal map takes $z^{pk}t^{-i}\mapsto upk z^{pk-1}\nabla z t^{-i}$ and the top horizontal map takes $z^kE(z)^it^{-i}$ to $u'pkz^{k-1}E(z)^{i-1} \nabla z t^{-i}$, with units $u,u'\in 1+p\bZ_p$. Since $\varphi(z^kE(z)^it^{-i}) = z^{pk}$, we get $\varphi^\nabla(z^{k-1}E(z)^{i-1}\nabla z t^{-i}) = \frac{u}{u'} z^{pk-1}\nabla z t^{-i}$, which reduces to the claim mod $p$.
\end{proof}

Note that the first and third identities above imply that in the square for $\bF_p(i)(\bZ_p)$, $(\can-\varphi)(z^kE(z)^it^{-i})$ has leading term (in $\rF$-filtration) $z^{k+i}t^{-i}$ if $k+i<pk$, i.e. $k>\frac{i}{p-1}$. Similarly, $(\can-\varphi^\nabla)(z^{k-1}E(z)^{i-1}\nabla z t^{-i})$ has leading term $z^{k+i-2}\nabla z t^{-i}$ if $k+i-1 <pk$, i.e. $k>\frac{i-1}{p-1}$. So in the subcomplex formed by the square
\[
\begin{tikzcd}
  \rF^{>\frac{i}{p-1}}\rN^{\geq i} \Prismhat^{(1)}_{\bZ_p/\bZ_p\llbracket z\rrbracket}\{i\}/p \rar{\nabla}\dar{\varphi} & \rF^{>\frac{i-1}{p-1}}\gr^k_{\rF} \rN^{\geq i} \Prismhat^{(1),\nabla}_{\bZ_p/\bZ_p\llbracket z\rrbracket}\{i\}/p\dar{\varphi^\nabla} \\ 
  \rF^{>\frac{pi}{p-1}}    \gr^{pk}_{\rF} \Prismhat^{(1)}_{\bZ_p/\bZ_p\llbracket z\rrbracket}\{i\}/p \rar{\nabla} & \rF^{>\frac{p(i-1)}{p-1}}\gr^{pk}_{\rF} \Prismhat^{(1),\nabla}_{\bZ_p/\bZ_p\llbracket z\rrbracket}\{i\}/p.
\end{tikzcd}
\]
the vertical maps are isomorphisms, and we may thus truncate the original square for $\bF_p(i)(\bZ_p)$ to a quasi-isomorphic one of the form
\[
\begin{tikzcd}
  \rF^{\leq\frac{i}{p-1}}\rN^{\geq i} \Prismhat^{(1)}_{\bZ_p/\bZ_p\llbracket z\rrbracket}\{i\}/p \rar{\nabla}\dar{\varphi} & \rF^{\leq\frac{i-1}{p-1}}\gr^k_{\rF} \rN^{\geq i} \Prismhat^{(1),\nabla}_{\bZ_p/\bZ_p\llbracket z\rrbracket}\{i\}/p\dar{\varphi^\nabla} \\ 
  \rF^{\leq\frac{pi}{p-1}}    \gr^{pk}_{\rF} \Prismhat^{(1)}_{\bZ_p/\bZ_p\llbracket z\rrbracket}\{i\}/p \rar{\nabla} & \rF^{\leq\frac{p(i-1)}{p-1}}\gr^{pk}_{\rF} \Prismhat^{(1),\nabla}_{\bZ_p/\bZ_p\llbracket z\rrbracket}\{i\}/p.
\end{tikzcd}
\]
We now compute its cohomology for specific values of $i$.

\begin{exm}[weight $i=0$]
  \label{exm:weight0}
  For $i=0$, the truncated square takes the form
  \begin{center}
  \begin{tikzcd}
    \bF_p \cdot 1\dar\rar & 0\dar \\
    \bF_p \cdot 1\rar & 0.
  \end{tikzcd}
  \end{center}
  Here the vertical map $\can-\varphi$ is zero, and we obtain  $H^0(\bF_p(0)(\bZ_p))\cong \bF_p$ and $H^1(\bF_p(0)(\bZ_p))\cong \bF_p$, both generated by classes represented by $1$. In fact, $1$ is in the kernel of $\nabla$ (since it is actually equalized by the two bottom coboundary maps in the cosimplicial diagram giving rise to $\nabla$), and so these classes are also represented by $1$ in the untruncated square. We denote the generator of $H^0$ by $1$, and the generator of $H^1$ by $\partial$ (compare Remark \ref{rem:products}(2)).
\end{exm}

\begin{exm}[weight $1 \leq i \leq p-2$]
  \label{exm:weighti}
  For $1\leq i\leq p-2$, the truncated square takes the form
  \begin{center}
  \begin{tikzcd}
    \bF_p \cdot E(z)^it^{-i}\dar\rar & 0\dar \\
    \bF_p \cdot \{t^{-i},\ldots, z^i t^{-i}\}\rar & \{\nabla z t^{-i},\ldots, z^{i-2}\nabla z t^{-i}\}.
  \end{tikzcd}
  \end{center}
  We have $(\can-\varphi)(E(z)^i t^{-i}) = -t^{-i}$ mod $\rF^{\geq 1}$. So this means that the homology with respect to the vertical differentials looks like
  \begin{center}
  \begin{tikzcd}
    0 \rar & 0 \\
    \bF_p \cdot \{zt^{-i},\ldots,z^i t^{-i}\}\rar & \{\nabla z t^{-i},\ldots, z^{i-2}\nabla z t^{-i}\}.
  \end{tikzcd}
  \end{center}
  with horizontal differential induced by $\nabla$. Since $\nabla(z^kt^{-i}) = kz^{k-1}\nabla z t^{-i}$ mod $\rF^{\geq k+1}$, the bottom horizontal differential is surjective, and its kernel is generated by $z^it^{-i}$. This means that $\bF_p(i)(\bZ_p)$ has cohomology concentrated in $H^1$, generated by a class represented in the original untruncated square by $z^it^{-i}$ in the bottom left corner plus something in the top right corner. We denote this class in $H^1(\bF_p(i)(\bZ_p))$ for $1\leq i\leq p-2$ by $\gamma_i$.
\end{exm}

\begin{exm}[weight $i=p-1$]
  \label{exm:weightpminus1}
  For $i=p-1$, the truncated square takes the form
  \begin{center}
  \begin{tikzcd}
    \bF_p \cdot \{E(z)^{p-1}t^{-p+1},zE(z)^{p-1}t^{-p+1}\}\rar\dar & 0\dar\\
    \bF_p \cdot \{t^{-p+1},\ldots,z^pt^{-p+1}\} \rar & \bF_p \cdot \{\nabla z t^{-p+1},\ldots, z^{p-3} \nabla z t^{-p+1}\}.
  \end{tikzcd}
  \end{center}
  We have $(\can-\varphi)(E(z)^{p-1}t^{-p+1}) = t^{-p+1}$ mod $\rF^{\geq 1}$, and $(\can-\varphi)(zE(z)^{p-1}t^{-p+1}) = z^pt^{-p+1} - z^pt^{-p+1} = 0$. So $H^0(\bF_p(p-1)(\bZ_p)) = \bF_p$, generated by a class represented by $zE(z)^{p-1}t^{-p+1}$. The lift to the untruncated complex is a priori given by $zE(z)^{p-1}t^{-p+1}$ plus terms in $\rF^{\geq 2}$, but since $zE(z)^{p-1}t^{-p+1}$ is in the kernel of $(\can-\varphi)$ in the untruncated complex and $(\can-\varphi)$ takes $\rF^{>1}$ isomorphically to $\rF^{>p}$, the representative in the untruncated complex is also given by $zE(z)^{p-1}t^{-p+1}$. We denote this class by $v_1$, since it detects the chromatic $v_1$ (compare \cite[Section 6, in particular Theorem 6.3]{kzpn}).

  Since the bottom map $\nabla$ takes $z^kt^{-p+1}$ to $kz^{k-1}\nabla z t^{-p+1}$ mod $\rF^{\geq k+1}$, it takes the span of $z^kt^{-p+1}$ for $1\leq k\leq p-2$ isomorphically to the bottom right term. So we obtain $H^2(\bF_p(p-1)(\bZ_p))=0$ and $H^1(\bF_p(p-1)(\bF_p))\cong \bF_p^2$, with generators represented by $z^{p-1}t^{-p+1}$ and $z^{p}t^{-p+1}$ in the bottom left corner. Note that by the description of product structures in Remark \ref{rem:products}, the latter class agrees with $\partial v_1$. The former can be lifted to some representative in the untruncated complex given by $z^{p-1}t^{-p+1}$ in the bottom left corner and some element in the top right corner. We will refer to this element by $\gamma_{p-1}$.
\end{exm}

\begin{exm}[weight $i=p$]
  For $i=p$, the truncated square takes the form
  \begin{center}
  \begin{tikzcd}
    \bF_p \cdot \{E(z)^pt^{-p},zE(z)^pt^{-p}, [z^2E(z)^pt^{-p}]\}\rar\dar & \bF_p \cdot E(z)^{p-1}\nabla z t^{-p}\dar\\
    \bF_p \cdot \{t^{-p},\ldots,z^{p+1}t^{-p}, [z^{p+2}t^{-p}]\} \rar & \bF_p \cdot \{\nabla z t^{-p},\ldots, z^{p-1} \nabla z t^{-p}\}.
  \end{tikzcd}
  \end{center}
  where the terms in brackets are present only in the case $p=2$. To study the homology of this double complex we first consider the vertical differentials. We have $(\can-\varphi)(E(z)^pt^{-p}) = -t^{-p}$ mod $\rF^{\geq 1}$, and $(\can-\varphi)(zE(z)^pt^{-p}) = -z^pt^{-p}$ mod $\rF^{\geq p+1}$. In the case $p=2$ we also have $(\can-\varphi)(z^2E(z)^2t^{-2}) = 0$, and in fact this element is $v_1^2$. In the right column, we have $(\can-\varphi^\nabla)(E(z)^{p-1}\nabla zt^{-p}) = z^{p-1}\nabla z t^{-p} - z^{p-1}\nabla z t^{-p}=0$. So the vertical homology looks like
  \begin{center}
  \begin{tikzcd}
    \bF_p \cdot \{[z^2E(z)^pt^{-p}]\}\rar & \bF_p \cdot E(z)^{p-1}\nabla z t^{-p}\\
    \bF_p \cdot \{zt^{-p},\ldots,z^{p-1}t^{-p},z^{p+1}t^{-p}, [z^{p+2}t^{-p}]\} \rar & \bF_p \cdot \{\nabla z t^{-p},\ldots, z^{p-1} \nabla z t^{-p}\}.
  \end{tikzcd}
  \end{center}
  Since the top left class for $p=2$ is a representative for $v_1^2$, it is a permanent cycle, so the top horizontal differential is zero. The bottom horizontal differential is induced by $\nabla$, and takes $z^kt^{-p}$ to $kz^{k-1}\nabla z t^{-p}$ mod $\rF^{\geq k+1}$. Since it also vanishes on $z^{p+1}t^{-p}$ (and $z^{p+2}t^{-p}$ for $p=2$) for filtration reasons, the horizontal homology looks like
  \begin{center}
  \begin{tikzcd}
    \bF_p \cdot \{[z^2E(z)^pt^{-p}]\} & \bF_p \cdot E(z)^{p-1}\nabla z t^{-p}\\
    \bF_p \cdot \{z^{p+1}t^{-p}, [z^{p+2}t^{-p}]\}  & \bF_p \cdot \{z^{p-1} \nabla z t^{-p}\}.
  \end{tikzcd}
  \end{center}
  We conclude that $H^0(\bF_p(p)(\bZ_p))=0$ ($\bF_2$ if $p=2$), $H^1(\bF_p(p)(\bZ_p)=\bF_p^2$ ($\bF_2^3$ if $p=2$) and $H^2(\bF_p(p)(\bZ_p)) = \bF_p$. The generator arising from $E(z)^{p-1}\nabla z t^{-p}$ lifts to a cycle in the untruncated complex purely supported in the top right corner, with leading term $E(z)^{p-1}\nabla z t^{-p}$ plus elements from $\rF^{\geq 2}$. We denote this element of $H^1(\bF_p(p)(\bZ_p))$ by $\lambda_1$. By Remark \ref{rem:products}(2),the cycle in the bottom right corner generating $H^2(\bF_p(p)(\bZ_p))$ is then a representative of $\partial \lambda_1$. The generator $z^{p+1}t^{-p}$ in the bottom left corner is a representative of $v_1\gamma_1$, since $\can(zE(z)^{p-1}t^{-p})=z^pt^{-p}$. And finally, the element $z^{p+2}t^{-p}$ in the case $p=2$ agrees with $\partial v_1^2$.
\end{exm}

Having named $\partial, v_1$ and $\lambda_1$, we can compute $\bF_p(*)(\bZ_p)/v_1$.

\begin{prop}
  \leavevmode
  \begin{enumerate}
    \item $\bF_p(*)(\bZ_p)/v_1$ in weights $*\geq p+1$ is $0$.
    \item $\bF_p(*)(\bZ_p)/v_1$ in weight $*=p$ has $H^0=0$, $H^1\cong \bF_p$ generated by $\lambda_1$, and $H^2\cong \bF_p$ generated by $\partial\lambda_1$.
    \item $\bF_p(*)(\bZ_p)/v_1$ in weights $*=i$ for $1\leq i\leq p-1$ has $H^0=H^2=0$ and $H^1\cong \bF_p$ generated by $\gamma_i$.
    \item $\bF_p(*)(\bZ_p)/v_1$ in weight $0$ has $H^0\cong \bF_p$ generated by $1$ and $H^1\cong \bF_p$ generated by $\partial$.
  \end{enumerate}
\end{prop}
\begin{proof}
  Since $v_1$ acts by multiplication with $zE(z)^{p-1}t^{-p+1}$ on the top row of the square and $z^p t^{-p+1}$ on the bottom row of the square, the mod $v_1$ reduced square for $\bF_p(*)(\Z_p)/v_1$ in weight $*=i$ for $i\geq p-1$ takes the form
  \[
    \begin{tikzcd}
      \bF_p \cdot E(z)^it^{-i} \rar\dar & \bF_p \cdot E(z)^{i-1}\nabla z t^{-i}\dar\\
      \bF_p \cdot \{t^{-i},\ldots, z^{p-1}t^{-i}\} \rar & \bF_p\cdot \{\nabla z t^{-i}, \ldots z^{p-1}\nabla z t^{-i}\}
    \end{tikzcd}
  \]
  Note that $\can(E(z)^it^{-i}) = z^it^{-i} = 0$ if $i\geq p$, and also $\can(E(z)^{i-1}\nabla z t^{-i}) = z^{i-1}\nabla z t^{-i}=0$ if $i\geq p+1$. Since $\varphi(E(z)^it^{-i}) = t^{-i}$ and $\varphi^\nabla(E(z)^{i-1}\nabla z t^{-i}) = z^{p-1}\nabla z t^{-i}$ (plus terms of higher $\rF$-filtration, but these vanish here in the quotient by $v_1$), the homology with respect to the vertical differentials takes the form
  \[
    \begin{tikzcd}
      0 \rar & 0\\
      \bF_p \cdot \{z t^{-i},\ldots, z^{p-1}t^{-i}\} \rar & \bF_p\cdot \{\nabla z t^{-i}, \ldots z^{p-2}\nabla z t^{-i}\}.
    \end{tikzcd}
  \]
  Since $\nabla(z^kt^{-i}) = kz^{k-1}\nabla z t^{-i}$ mod $\rF^{\geq k+1}$, this remaining horizontal differential is an isomorphism, and the first claim follows.

  In weights $i=p-1$ and $i=p$, we get the same square, but $\can$ is nontrivial. We have $(\can-\varphi)(E(z)^it^{-i}) = -t^{-i} + z^it^{-i}$ and $(\can-\varphi^\nabla)(E(z)^{i-1}\nabla z t^{-i}) = z^{i-1} \nabla z t^{-i} - z^{p-1}\nabla z t^{-i}$ mod $\rF^{\geq p+1}$. For $i=p$, this means the homology with respect to the vertical differentials takes the form
  \[
    \begin{tikzcd}
      0 \rar & \bF_p\cdot \lambda_1 \\
      \bF_p \cdot \{z t^{-i},\ldots, z^{p-1}t^{-i}\}\rar  & \bF_p\cdot \{\nabla z t^{-i}, \ldots z^{p-2}\nabla z t^{-i}, \partial\lambda_1\},
    \end{tikzcd}
  \]
  and the second claim follows. Analogously, for $i=p-1$ the homology with respect to the vertical differentials takes the form
  \[
    \begin{tikzcd}
      0 \rar & 0 \\
      \bF_p \cdot \{z t^{-i},\ldots, z^{p-1}t^{-i}\} \rar & \bF_p\cdot \{\nabla z t^{-i}, \ldots z^{p-3}\nabla z t^{-i}\},
    \end{tikzcd}
  \]
  and the horizontal differential leaves just a single class represented by $z^{p-1}t^{-p+1}$ in the lower left corner. This is $\gamma_{p-1}$ described in Example \ref{exm:weightpminus1} above. This proves the $i=p-1$ part of the third claim. For weight $i<p-1$, the mod $v_1$ reduced version looks just like the unreduced version. So the remaining statements follow from Examples \ref{exm:weighti} and \ref{exm:weight0}.
\end{proof}

\begin{cor} 
  \label{cor:zp-syntomic}
  The cohomology of $\bF_p(*)(\bZ_p)$ admits a basis given by the elements $v_1^k$, $v_1^k\partial$, $v_1^k\gamma_i$ for $1\leq i\leq p-1$, $v_1^k \lambda_1$ and $v_1^k\partial\lambda_1$.
\end{cor}
\begin{proof}
  This follows immediately from the $v_1$ Bockstein spectral sequence, since all the elements described above as generators of the cohomology of $\bF_p(*)(\bZ_p)/v_1$ lift to the cohomology of $\bF_p(*)(\bZ_p)$ by construction.
\end{proof}

\subsection{Computation of $H^2(\F_p (*) (\Z/p^n))$}\label{sec:final}
We now study $\bF_p(*)(\bZ/p^n)$. Under the map $\bZ_p\to \bZ/p^n$, the classes introduced above map to classes in $H^*(\bZ_p(*)(\bZ/p^n))$. The main result of this paper is the following:
\begin{thm}
  \label{thm:mainthm}
  $H^2(\bF_p(*)(\bZ/p^n))$ admits a basis given by the elements $v_1^k\partial \lambda_1$ with $k\leq p^{n-2}-1$.
\end{thm}

In light of \Cref{thm:surjective} and \Cref{cor:zp-syntomic}, it will suffice to check $v_1^{p^{n-2}}\partial\lambda_1=0$ as well as the nonvanishing of $v_1^k\partial\lambda_1$ for $k<p^{n-2}$.

To avoid explicitly dealing with the terms of the square involving $\nabla$, we use the product structure to reduce the vanishing of $v_1^{p^{n-2}}\partial\lambda_1$ to a claim about $(\can-\varphi)$. We first check the following:

\begin{lem}
  In the square computing $\bF_p(*)(\bZ/p^n)$, $v_1^{p^{n-2}}\partial$ may be represented by an element supported in the top right corner of the square.
\end{lem}
\begin{proof}
  Recall that $v_1^{p^{n-2}}\partial$ can be represented by $z^{p^{n-1}}t^{-p^{n-1}+p^{n-2}}$ in the bottom left corner of the square. The claim is thus equivalent to proving that $z^{p^{n-1}}t^{-p^{n-1}+p^{n-2}}$ is in the image of $\can-\varphi$.  This follows from the following lemma for $j=0$, since $f_0=z^n$.
\end{proof}

\begin{lem}
  For $i=p^{n-1}-p^{n-2}$, any term of the form $u z^{p^{n-1}-(n-j)p^j}f_jt^{-i}$ for $j\leq n$ and arbitrary $u$, is in the image of $\can-\varphi$ mod $p$. 
\end{lem}
\begin{proof}
  Since elements of high $\rF$-filtration in $\Prismhat^{(1)}_{(\bZ/p^n)/\bZ_p\llbracket z\rrbracket}$ have high Nygaard filtration, specifically $\rF^{\geq nk}\subseteq \rN^{\geq k}$ by \cite[Proposition 3.34 and Remark 3.35]{kzpn}, the $2$-term complex
  \[
    \can-\varphi: \rN^{\geq i} \Prismhat^{(1)}_{R/\bZ_p\llbracket z\rrbracket}\{i\}/p
    \to \Prismhat^{(1)}_{R/\bZ_p\llbracket z\rrbracket}\{i\}/p
  \]
  is quasi-isomorphic to the truncated one of the form
  \[
    \can-\varphi: \rF^{<in}  \rN^{\geq i} \Prismhat^{(1)}_{R/\bZ_p\llbracket z\rrbracket} \{i\}
\to \rF^{<in} \Prismhat^{(1)}_{R/\bZ_p\llbracket z\rrbracket} \{i\}
  \]
  So it suffices to decide that the terms in question lie in the image modulo $\rF$-filtration $\geq in$.

  For $j=n$ or $n-1$, we have $f_j\in \rF^{\geq np^{n-1}}$, which is already beyond that filtration. It therefore suffices to argue that for $j<n-1$, $uz^{p^{n-1}-(n-j)p^j}f_jt^{-i}$ agrees modulo $p$ and the image of $\can-\varphi$ with a term of the same form but with $j+1$ instead of $j$. And indeed we have:
\[
  \can(u E(z)^{p^{n-1}-p^{n-2}-p^j} z^{p^{n-2} - (n-j-1)p^j} f_j t^{-i}) = uz^{p^{n-1}-(n-j)p^j} f_j t^{-i},
\]
using that $j\leq n-2$ and $p^{n-2-j} \geq n-1-j$, and
\[
  \varphi(uE(z)^{p^{n-1}-p^{n-2}-p^j}z^{p^{n-2}-(n-j-1)p^j}f_jt^{-i}) = \varphi(u)\lambda_j z^{p^{n-1}-(n-j-1)p^{j+1}} f_{j+1}t^{-i},
\]
  finishing the inductive step. Here $\lambda_j$ is a unit in $\bZ_p\llbracket z\rrbracket$ appearing in the explicit formulas for $\varphi(f_jt^{-p^j})$ (\cite[Lemma 3.26]{kzpn}), and not to be confused with the explicit generator $\lambda_1\in H^1(\bF_p(p)(\bZ_p))$ from above.
\end{proof}

\begin{cor}\label{cor:vanishing}
We have  $v_1^{p^{n-2}}\partial\lambda_1=0$
\end{cor}
\begin{proof}
  As both $v_1^{p^{n-2}}\partial$ and $\lambda_1$ are represented by an element supported in the top right corner, their product is zero by \Cref{rem:products}(3).
\end{proof}

\begin{prop}\label{prop:nonvanishing}
We have $v_1^{p^{n-2}-1}\partial\lambda_1\neq 0$.
\end{prop}
\begin{proof}
  By \cite[Theorem 1.1]{HLS}, $\bF_p(*)(R) / v_1^{p^{n-2}}$ depends only on the derived mod $p^n$ reduction $R\otimes_{\bZ} \bZ/p^n$. In particular, we obtain the same value for $R$ given by $\bZ/p^n$ and $R'$ given by an exterior algebra $\Lambda_{\bZ_p}(\varepsilon)$ on a degree $1$ generator (viewed as animated ring). If $v_1^{p^{n-2}-1}\partial\lambda_1=0$, then this would imply that also $H^2(\bF_p(*)(\Lambda_{\bZ_p}(\varepsilon))/v_1^{p^{n-2}})$ vanishes in weight $*=p+(p-1)(p^{n-2}-1)$. But the latter has $H^2(\bF_p(*)(\bZ_p)/v_1^{p^{n-2}})$ as retract, which is nonzero by Corollary \ref{cor:zp-syntomic}.
\end{proof}

Together this proves Theorem \ref{thm:mainthm}.

\bibliographystyle{alpha}
\bibliography{bibliography}

\end{document}